\documentclass[11pt, oneside, 
a4paper]{article}
\usepackage[ansinew]{inputenc}
\usepackage[english]{babel}
\usepackage{amsbsy}
\usepackage{amscd}
\usepackage{amsfonts}
\usepackage{amsmath}
\usepackage{amsopn}
\usepackage{amssymb}
\usepackage{amstext}
\usepackage{amsthm}
\usepackage{amsxtra}
\usepackage{color}
\usepackage{float}
\usepackage{graphicx}
\usepackage{latexsym}
\usepackage{srcltx}
\usepackage{times}
\usepackage{url}

\newcommand{\suces}[3]{{#1_{#2},\hdots,#1_{#3}}}

\newcommand{\realamp}{\R\cup\{-\infty\}}

\newcommand{\T}{\mathbb{T}}

\newcommand{\m}{\medskip}

\newcommand{\ui}{{\underline{i}}}

\newcommand{\uk}{{\underline{k}}}

\newcommand{\ot}{{\overline{\T}}}

\newcommand{\N}{\mathbb{N}}

\newcommand{\Q}{\mathbb{Q}}
\newcommand{\R}{\mathbb{R}}

\newcommand{\MM}{\mathcal{M}}

\renewcommand{\int}{\operatorname{int}}

\newcommand{\row}{\operatorname{row}}
\newcommand{\rows}{\operatorname{rows}}
\newcommand{\col}{\operatorname{col}}
\newcommand{\win}{\operatorname{win}}
\newcommand{\sol}{\operatorname{sol}}
\newcommand{\dif}{\operatorname{dif}}
\newcommand{\card}{\operatorname{card}}

\newtheorem{thm}{Theorem}

\newtheorem{dfn}[thm]{Definition}

\newtheorem{ex}[thm]{Example}
\newtheorem{rem}[thm]{Remark}

\begin{document}
\title{An algorithm to describe the solution set of  any tropical linear system  $A\odot x=B\odot x$}
\author{E. Lorenzo\thanks{Partially supported by a La Caixa grant and a FPU grant.} \\ Dpto. de Matem\'{a}tica Aplicada 2\\Facultad de Matem\'{a}ticas y Estad\'{\i}stica\\Universidad Polit\'{e}cnica de Catalu\~{n}a\\\texttt{elisa.lorenzo@gmail.com} \and M.J. de la Puente\thanks{Partially supported by UCM research group 910444. Corresponding author.}\\Dpto. de Algebra\\Facultad de Matem\'{a}ticas\\Universidad Complutense\\
\texttt{mpuente@mat.ucm.es}\\}

\maketitle
\begin{abstract}

AMS class.: 15--04; 15A06; 15A39.

Keywords and phrases: tropical linear system, algorithm.

An algorithm to give an explicit description of all the solutions to  any tropical linear system  $A\odot x=B\odot x$ is presented. The given system is converted into a finite (rather small) number  $p$ of pairs $(S,T)$ of classical linear systems: a system  $S$ of  equations  and a system $T$ of inequalities. The notion, introduced here,  that makes $p$  small, is called compatibility. The particular feature of both $S$  and $T$ is that each  item (equation or inequality) is bivariate, i.e., it involves exactly two variables;   one variable with coefficient $1$, and the other one with $-1$. $S$ is solved by Gaussian elimination. We explain how to solve $T$ by a method similar to Gaussian elimination. To achieve this, we introduce the notion of sub--special matrix. The procedure applied to $T$ is, therefore, called sub--specialization. 
\end{abstract}

\section{Introduction}
Consider the set $\realamp$, denoted $\T$ for short, endowed with
\emph{tropical addition} $\oplus$ and \emph{tropical multiplication} $\odot$, where
these operations are defined as follows:
    $$a\oplus b=\max\{a,b\},\qquad a\odot b=a+b,$$ for
    $a,b\in\T$. Here, $-\infty$ is the neutral element for tropical addition and   $0$ is the neutral element for tropical multiplication.
    Notice that
    $a\oplus a=a$, for all $a$, i.e., tropical addition is \emph{idempotent}. Notice also that $a$ has no inverse with respect to $\oplus$.
We will write $\oplus$ or $\max$, (resp. $\odot$ or $+$) at our convenience.
In this  paper we will use the adjective \emph{classical} as opposed to \emph{tropical}. Most definitions in tropical mathematics just  mimic the classical ones.
Very often, working with  $(\T,\oplus, \odot)$ leads to  working with $\min$, which will be denoted $\oplus'$.

Given matrices $A,B\in\MM_{m\times n}(\T)$, we want to  describe all   $x\in\T^n$ such that
$A\odot x=B\odot x$. This means
\begin{equation*}
\max\{a_{ij}+x_j:\ 1\le j\le n\}=\max\{b_{ij}+x_j:\ 1\le j\le n\}, \quad i=1,2,\ldots,m.
\end{equation*}
{Of course, $x_j=-\infty$, all $j=1,2,\ldots,n$ is a solution (the \emph{trivial solution}).}
 Using the notions of \emph{winning pairs}, \emph{compatibility} and \emph{win sequence} (introduced in this paper; see. definitions \ref{dfn:winning_pair}, \ref{dfn:compatible} and \ref{dfn:maximal}), the given problem is reduced to solving a finite, {rather small}, number of pairs $(S,T)$ of classical linear systems: a system $S$ of equations and a system $T$ of inequalities, each  item (equation or inequality) being \emph{bivariate}, i.e., it involves exactly two variables,  one with coefficient $1$, and another with coefficient $-1$.  Of course, $S$ can be easily solved by Gaussian elimination. On the other hand, the fact that $T$ consists of bivariate inequalities allows us to suggest  a Gaussian--like procedure  to solve system $T$.   More precisely, by certain {row operations (see p. \pageref{dfn:row_oper})},  we do not triangulate a coefficient matrix for $T$, but nearly so. What we  do is to  transform such a  matrix {into two matrices}, one of which is a \emph{sub--special matrix} (see definition \ref{dfn:special}). {From these two matrices}  the solution set to the   $T$ is directly read. We call this procedure \emph{sub--specialization}, (see remark \ref{rem:sub_sp}). Notice that we need not  use the simplex algorithm or other well--known ones to solve $T$.

 Compatibility of winning pairs (see definition \ref{dfn:compatible}) turns out  a very handy necessary condition to work with. Indeed, two given winning pairs yield  {some}  bivariate linear inequalities. Imagine that {two of} these are  $x_1-x_2+a\le0$ and $-x_1+x_2-b\le0$, for some $a,b\in \T$. Now the compatibility of the winning pairs guarantees that these inequalities  hold simultaneously, i.e., {they can be concatenated into} $x_1+a\le x_2\le x_1+b$, so that $a\le b$. Therefore, by requiring  compatibility of winning pairs, all we are doing is ruling out, at the very beginning, wasting out time with systems of inequalities which, for sure,  the only solution is trivial.
 \m
 The problem $A\odot x= B\odot x$ has been addressed before.  Indeed,  in \cite{Butkovic_Hegedus}, it is proved that the solution set  can be finitely generated. In \cite{Butkovic_Zimmer}, a strongly polynomial algorithm is found which either finds a solution or  tells us that no solution exists. In \cite{Baccelli} sec. 3.5, all the solutions are computed by a technique called symmetrization and resolution of balances. In \cite{Butkovic_Hegedus}, generators for the solution set are computed. The idea of finding (a minimal family of) generators is pursued in \cite{Sergeev_Wagneur, Truffet, Wagneur_al} for the equivalent problem $A\odot x\le B\odot x$. An iterative method in presented in \cite{Cuninghame_B_2003}
 for another  equivalent problem, namely $A\odot x= B\odot y$, where $x$ and $y$ are unknown. Also, there is a technique in \cite{Walkup} to solve the problem $A\odot x\oplus a= B\odot x\oplus b$, relying on a recursive formulation of the closure operator (also called Kleene star operator) on matrices. In \cite{Gondran} ch. 4, the closely related   problem $A\odot x\oplus b=x$ (similar to the classical Jacobi iterative method) is solved using Kleene stars.

 We present the solution set  to $A\odot x=B\odot x$ as a finite union of sets, which are (attending to their presentation) obviously convex. This, of course, agrees with the cellular decomposition  in \cite{Develin}. Convexity issues are also studied in  \cite{Sergeev_S_B}. The recent paper \cite{Truffet} addresses the problem  $A\odot x\le B\odot x$. The drawback of the algorithm presented in \cite{Truffet} is, in our opinion,  that this algorithm calls for the calculation of certain Kleene stars, and this is not an easy task.

 We do not describe the solution set by generators. We do not use Kleene stars. Instead, our method provides a description of each convex piece of the solution set  by parameters and bivariate linear inequality relations among them, in a reduced form.


\m

Let $m,n,s\in\N$ be given.
The following are kindred problems in tropical linear algebra:
\begin{itemize}
\item  $P1$: $A\odot x=0$, \label{prb:uno}
\item  $P2$: $A\odot x=b$, \label{prb:dos}
\item  $P3$: $A\odot x\le b$, \label{prb:tres}
\item  $P4$: $A\odot x=B\odot x$, \label{prb:cuatro}
\item  $P5$: $A\odot x\le B\odot x$,  \label{prb:cinco}
\item  $P6$: $C\odot x= D\odot y$, \label{prb:seis}
\item  $P7$: $A\odot x\oplus a= B\odot x\oplus b$.\label{prb:siete}
\end{itemize}
Here the data are matrices $A,B\in\MM_{m\times n}(\T)$, $C\in\MM_{s\times n}(\T)$, $D\in\MM_{s\times m}(\T)$  and vectors $a,b$ over $\T$, and \textbf{the $j$th problem} is \textbf{computing all} vectors $x\in\T^n$, $y\in\T^m$,  such that $Pj$ holds.   By \textbf{deciding the $j$th problem}
we mean either finding one solution or declaring that the problem has no (non--trivial) solution. Some references  since 1984 are \cite{Baccelli, Butkovic_Hegedus, Cuninghame_New, Cuninghame_B_2003, Sergeev_Wagneur, Wagneur_al, Walkup}.  Earlier books and papers can be found there.

\m

Of course, $x=-\infty$, $y=-\infty$ are solutions to $P3,P4,P5$ and $P6$. These are the \emph{trivial solutions}.

Only if the vector $b$ is real,  problem $P2$ reduces to $P1$. More generally,
one must realize that, contrary to classical linear algebra, problems $P7$ and $P4$
do not reduce to problem $P2$ or $P1$, because there are no inverses for tropical addition, so there is
no tropical analogue for the matrix $-B$.
Nevertheless, there are well--known connections among these problems, i.e., being able to
solve some of them is equivalent to being able to solve some other.

\m

We need some notations:
\begin{itemize}
\item For $c,d\in\T$, $c\oplus' d$ means $\min\{c,d\}$ and $c\odot' d$ means $c+d$.
\item For $c,d\in\T^n$, $c\odot' d^T$ means $\min\{c_1+d_1,c_2+d_2,\ldots,c_n+d_n\}$.
\item If $A=(a_{ij})\in\MM_{m\times n}(\R)$ then  $A^*=(-a_{ji})$ is the \emph{conjugate matrix}.
\end{itemize}

The relationship among these problems is as follows:
\begin{itemize}
\item \emph{Deciding $P3$ is possible, if $A$ is real.}

Indeed,  $x^\#=A^*\odot'b$ is a solution (called \emph{principal solution}) and
$x\le x^\#$ if and only if $x$ is a solution; see \cite{Cuninghame_New}, p.~ 31; in \cite{Baccelli}
this process is called \emph{residuation}.

\item \emph{Deciding $P3$ helps with deciding $P2$, if $A$ is real.}

Indeed, $P2$ might be incompatible but,
if it has a solution, then $x^\#$ is the greatest one;  see \cite{Cuninghame_New}, p.~ 31.

\item \emph{Deciding $P6$ implies deciding $P2$.}

Given $A$ and $b$, we decide $A\odot x=I\odot y$. For each pair of solutions $x,y$, if any, we set $y=b$, if possible.

\item \emph{Deciding $P4$ is equivalent to deciding $P6$.}

Suppose $x$ is a solution to $P4$ and write $A\odot x=y$.  Concatenating matrices, write
$C=\left[A \atop B\right]\in\MM_{2m\times n}(\T)$, $D=\left[I \atop I\right]\in\MM_{2m\times n}(\T)$,
where $I$ is the tropical identity matrix, so that $C\odot x=D\odot y$. Therefore, if we can decide $P6$, then we can decide $P4$.

Suppose now $x,y$ are  solutions to $P6$ and write $z=\left[x \atop y\right]$,
$A=\left[
C,-\infty\right]$, $B=\left[
-\infty,D\right]$
so that $A\odot z=B\odot z$. Therefore, if we can decide $P4$, then we can decide $P6$.

\item \emph{$P4$ and $P5$ are equivalent.}

$A\odot x=B\odot x$ is equivalent to $A\odot x\le B\odot x$ and $A\odot x\ge B\odot x$. On the other hand, $A\odot x\le B\odot x$ is equivalent to $(A\oplus B)\odot x=B\odot x$.

\item \emph{Deciding $P4$ implies deciding $P7$.}

We introduce a new scalar variable $z$ and write $A\odot x\oplus a\odot z= B\odot x\oplus b\odot z$.
 Concatenating matrices, write $t=\left[x \atop z\right]$, $C=[A,a]$, $D=[B,b]$ so that $C\odot t=D\odot t$.
 After solving $P4$, set $t_{n+1}=z=0$.
\end{itemize}

\section{The problem}

Given matrices $A,B\in\MM_{m\times n}(\T)$, we want to  describe all {non--trivial}  $x\in\T^n$ such that
\begin{equation}\label{eqn:igualdad}
A\odot x=B\odot x.
\end{equation}
Notations:
\begin{itemize}
\item $\textbf{A}=(\textbf{a}_{ij})$, $\textbf{B}=(\textbf{b}_{ij})$, with
$${\textbf{a}_{ij}=\begin{cases}
a_{ij}& \text{if $a_{ij}\ge b_{ij}$,}\\
-\infty&\text{otherwise,}
\end{cases}} \qquad{\textbf{b}_{ij}=\begin{cases}
b_{ij}& \text{if $a_{ij}\le b_{ij}$,}\\
-\infty&\text{otherwise.}
\end{cases}}$$
\item $M=A\oplus B=\textbf{A}\oplus \textbf{B}=(m_{ij})$ is the \emph{maximum matrix}.
\end{itemize}

Let  \begin{equation}\label{eqn:igualdad2}
\textbf{A}\odot x=\textbf{B}\odot x=M\odot x.
\end{equation}
Notice that (\ref{eqn:igualdad}) is equivalent to (\ref{eqn:igualdad2})
and (\ref{eqn:igualdad2})
is simpler than (\ref{eqn:igualdad}) because it involves fewer real coefficients.
Thus, we will assume that $A=\textbf{A}$ and $B=\textbf{B}$, (\emph{assumption 1}) in the following.

More notations:
\begin{itemize}
\item  $[n]=\{1,\ldots,n\}$, for $n\in\N$. \item For any $c\in \T$, $x=c\in \T^n$ means $x_j=c$,  for all $j\in [n]$.
\item  {For convenience,  we extend $\T$ to $\ot$, by adding $+\infty$. Eventually we will remove $+\infty$.}
\item If $k\in[m]$  and $j,l\in [n]$, then
$$\dif(M;j,l)_k=
\begin{cases}
m_{kj}-m_{kl},& \text{if $m_{kl}\neq -\infty$},\\
+\infty& \text{if $m_{kj}\neq-\infty$ and $m_{kl}=-\infty$},\\ 
\text{undetermined,}& \text{otherwise}.
\end{cases}$$   The undetermined case will never appear in the following. 
\item {$\Omega=\{j\in[n]:x_j=-\infty\}.$}
\end{itemize}

We will deal with  bivariate equalities and inequalities,  linear  on the $x_j$'s with coefficients in $\ot$.
Tautological linear equalities or inequalities will be always removed (\emph{assumption 2}).
A non--tautological linear equality not involving $\pm \infty$ will be reduced  to  an equivalent equation of the form $a=0$, for some  $a$, by the usual algebraic rules. {Of course, $-a=0$ is also possible}.
 A non--tautological linear inequality not involving $\pm \infty$ will be reduced  to  an equivalent inequality of the form $a\le0$, for some  $a$. These will be called  \textbf{normal forms}. Notice that normal forms have real coefficients.

 What will we do with certain normal forms $\phi$, if we know that $j\in\Omega$, i.e., $x_j=-\infty$? We will \textbf{remove and enlarge} as follows:\label{erasing}
 \begin{itemize}
 \item  If $\phi$ is $x_j-x_k+a\le0$, for some $k\neq j$ and $a\in\R$, then remove $\phi$.
 \item  If $\phi$ is $x_k-x_j+a\le0$, for some $k\neq j$ and $a\in\R$, then remove $\phi$  and set  $x_k=-\infty$, i.e., enlarge $\Omega$ with $k$.
 \end{itemize}

%
%
%
%
%

\begin{rem}
\begin{enumerate}
\item $x=-\infty$  satisfies (\ref{eqn:igualdad}). This is the \emph{trivial solution}.
\item If $\row(A,i)>\row(B,i)$ or $\row(A,i)<\row(B,i)$ for some $i\in [m]$, then $x=-\infty$ is the only solution to (\ref{eqn:igualdad}).
\item If $\row(A,i)=\row(B,i)$ for some $i\in [m]$,  then these two rows can be removed, so that $m$ can be decreased to $m-1$.
\item If $\col(A,j)=\col(B,j)=-\infty$ for some $j\in [n]$,  then no restriction is imposed on $x_j$.  Then these two columns  and $x_j$ can be removed, so that $n$ decreases to $n-1$.
\end{enumerate}
\end{rem}

We  will assume that $\row(A,i)\neq\row(B,i)$ (\emph{assumption 3}), $\row(A,i)\nless\row(B,i)$  (\emph{assumption 4}) and $\row(A,i)\ngtr\row(B,i)$ (\emph{assumption 5}), for all $i\in [m]$, and $\col(A,j)=\col(B,j)=-\infty$ (\emph{assumption 6}), for no $j\in [n]$, in the following.

\m
The sets in the next definition  are denoted $I,J,K,L$ in \cite{Butkovic_Hegedus}.

\begin{dfn}\label{dfn:winning_pair} For each $i\in [m]$,  let
\begin{enumerate}
\item $WA(i)=\{j: a_{ij}>b_{ij}\}$, $WB(i)=\{j: a_{ij}<b_{ij}\}$.
\item $E(i)=\{j: a_{ij}=b_{ij}\neq-\infty\}$, $F(i)=\{j: a_{ij}=b_{ij}=-\infty\}$.
\item $\win(i)=\left(WA(i)\times WB(i)\right)\cup \left(E(i)\times E(i)\right)\subset [n]\times [n]$. Each element of $\win(i)$ is called a \emph{winning pair}.
\end{enumerate}
\end{dfn}
For each $i\in [m]$, $WA(i)\cup WB(i)\cup E(i)\cup F(i)=[n]$ is a \emph{disjoint} union. {By  assumption 3,  $E(i)\cup F(i)\neq [n]$. By  assumptions 4 and 5,   $WA(i)\neq [n]\neq WB(i)$.} Moreover,  $\cap_{h=1}^mF(h)=\emptyset$, by  assumption 6.

\begin{ex}\label{ex:current}
Given
$$A=\left[\begin{array}{rrrr}
 3 &    7&    -1&  -\infty \\
     6&     7&  -\infty &  -\infty \\
     1&     0&     1&  -\infty \\
\end{array}\right], \quad
B=\left[\begin{array}{rrrr}
-\infty &  -\infty &  -\infty &     8\\
  -\infty &  -\infty &     5&     1\\
     1&     0&     1&     2\\
\end{array}\right],$$ we get
$$M=\left[\begin{array}{rrrr}
3&     7&    -1&     8\\
     6&     7&     5&     1\\
     1&     0&     1&     2\\
\end{array}\right].$$
Then $WA(1)=\{1,2,3\}$, $WB(1)=\{4\}$, $WA(2)=\{1,2\}$, $WB(2)=\{3,4\}$, $WB(3)=\{4\}$ and $E(3)=\{1,2,3\}$. Therefore,
$\win(1)=\{(1,4),(2,4),(3,4)\}$, $\win(2)=\{(1,3),(1,4),(2,3),(2,4)\}$, $\win(3)=\{(1,1),(2,2),(3,3)\}$.
\end{ex}

\begin{rem}
Suppose that  $\win(i)=\emptyset$. This is equivalent to either $\row(A,i)=-\infty$ or $\row(B,i)=-\infty$, but not both, by assumption 3. Say $\row(A,i)=-\infty$ and $b_{ij}\neq-\infty$ for some $j\in[n]$. If $A\odot x=B\odot x=y$, then $y_i=-\infty$ and $x_j=-\infty$. With this information in mind, we can remove the $i$--th rows, so that $m$ decreases to $m-1$. Therefore, we will assume that $\row(A,i)\neq -\infty$ and $\row(B,i)\neq -\infty$, for all $i\in[m]$ (\emph{assumption 7}). Then $\win(i)\neq\emptyset$, for all $i\in[m]$.
\end{rem}

Non--trivial solutions to (\ref{eqn:igualdad})  arise from winning pairs. Let us see how.
Recall that $M=A\oplus B=(m_{ij})\in \MM_{m\times n}(\T)$. {$M$ might not be real (see example \ref{ex:non_real})}.

\begin{dfn}\label{dfn:arises}
Consider $i\in [m]$ and $I\in \win(i)$. Let $x\in \T^n$, $y\in\T^m$ be any vectors satisfying
$A\odot x=B\odot x=y$ (in particular,  $$\row(A,i)\odot x=\row(B,i)\odot x=\row(M,i)\odot x=y_i).$$ We say that \emph{the solution $x$ to (\ref{eqn:igualdad}) arises from $I$} if
\begin{equation}\label{eqn:ineq}
m_{ij }+x_j  \le y_i,
\end{equation}  for all $j  \in [n]\setminus F(i)$, with
\textbf{equality for all $j\in |I|$.}
\end{dfn}

{Notice that $j\in F(i)$ if and only if $m_{ij}=-\infty$ and, in such a case,  the inequality  (\ref{eqn:ineq}) is tautological.  Otherwise,     $m_{ij}$  is real.}

\m
Suppose that $x$ arises from $I$ and
write $|I|=\{\ui_1,\ui_2\}$. Then, equality for all $j\in |I|$ means
$$m_{i{\ui_1} }+x_{\ui_1}=y_i=m_{i{\ui_2} }+x_{\ui_2},$$ whence we obtain one \emph{bivariate linear equation}
\begin{equation}\label{eqn:eqn}
x_{\ui_2}=\dif(M;\ui_1,\ui_2)_i+x_{\ui_1}.
\end{equation}
{Notice that $\dif(M;\ui_1,\ui_2)_i$ is real.}
In addition, (\ref{eqn:ineq}) amounts to, \emph{at most, $2n-2$  bivariate linear inequalities in the $x_j$'s.}
Notice that (\ref{eqn:eqn})  is tautological, when  $\ui_1=\ui_2$.

\setcounter{thm}{2}
\begin{ex} (Continued) Take $i=1$, $I=(1,4)\in\win(1)$ and suppose that a solution $x$ to (\ref{eqn:igualdad}) arises from $I$. This means $$3+x_1=8+x_4$$ (one bivariate equation) and  $$7+x_2\le y_1,\qquad -1+x_3\le y_1,$$ where $3+x_1=y_1$. Replacing $y_1$ by its value, we obtain two additional  bivariate inequalities. Altogether, (\ref{eqn:ineq}) becomes, in normal form,
\begin{align}
x_1-x_4-5&=0,\label{eqn:I}\\
-x_1+x_2+4&\le0,\label{eqn:II}\\
-x_1+x_3-4&\le0.\label{eqn:III}
\end{align}
\end{ex}


\setcounter{thm}{5}
\begin{rem}\label{rem:compatible}
Suppose $i, k\in [m]$, $i<k$, $I\in \win(i), K\in \win(k)$. Assume that the solution $x$ arises from $I$ and from $K$.  Then for all $\ui \in |I|$ and $\uk \in |K|$, we have
\begin{align}
 m_{i\ui  }+x_\ui  &=y_i,\label{eqn:rem21}\\
 m_{i\uk  }+x_\uk  &\le y_i,\label{eqn:rem22}\\
 m_{k\uk  }+x_\uk  &=y_k,\label{eqn:rem23}\\
 m_{k\ui  }+x_\ui  &\le y_k.\label{eqn:rem24}
\end{align}
Adding up,  $$m_{i\uk  }+m_{k\ui  }+x_\ui  +x_\uk  \le m_{i\ui  }+m_{k\uk  }+x_\ui  +x_\uk  =y_i+y_k,$$ whence
\begin{equation}\label{eqn:condicion}
m_{i\uk  }+m_{k\ui  }\le m_{i\ui  }+m_{k\uk  }.
\end{equation}
In other words,  the value of the $2\times 2$ tropical minor of $M$, denoted $M(i,k; \ui,\uk)$,
\begin{equation}\label{eqn:minor}
\left|\begin{array}{cc}
m_{i\ui  }&m_{i\uk  }\\
m_{k\ui  }&m_{k\uk  }\\
\end{array}\right|_{trop}=\max\{m_{i\ui  }+m_{k\uk  },\ m_{i\uk  }+m_{k\ui  }\},
\end{equation}
 is \emph{attained at the main diagonal}. One more way to {write this over $\ot$}  is
\begin{equation}\label{eqn:compat}
\dif(M;\ui,\uk)_k\le\dif(M;\ui,\uk)_i.
\end{equation}
\end{rem}
This remark leads to the following key definition.

\begin{dfn}\label{dfn:compatible}
 Consider $i, k\in [m]$, $i<k$, $I\in \win(i), K\in \win(k)$. We say that \emph{$K$ is compatible with $I$} if the value of (\ref{eqn:minor}) is \emph{attained at the main diagonal}, for all $\ui \in |I|$ and all $\uk \in |K|$. {Equivalently, if (\ref{eqn:compat}) holds in $\ot$, for all $\ui \in |I|$ and all $\uk \in |K|$.}
 \end{dfn}

\setcounter{thm}{2}
\begin{ex} (Continued) In our running example, take $i=1$, $k=2$, $I=(1,4)$ and $K=(1,3)$.  Then $K$ is compatible with $I$, since each tropical minor
$$\left|\begin{array}{cc}
3&3\\
6&6\\
\end{array}\right|_{trop}=9,\ \left|\begin{array}{cc}
3&-1\\
6&5\\
\end{array}\right|_{trop}=8,\ \left|\begin{array}{cc}
8&3\\
1&6\\
\end{array}\right|_{trop}=14,\ \left|\begin{array}{cc}
8&-1\\
1&5\\
\end{array}\right|_{trop}=13$$ attains its value at the main diagonal. {The inequalities (\ref{eqn:compat}) are $0\le0$, $1\le4$, $-5\le 5$ and $-4\le9$ in this case.} However, $K=(1,4)$  is not compatible with $I$, since the tropical minor
$$\left|\begin{array}{cc}
3&8\\
6&1\\
\end{array}\right|_{trop}=14$$ does not attain its value at the main diagonal.
\end{ex}

\setcounter{thm}{7}
\begin{rem} \label{rem:decrease} Suppose $i, k\in [m]$, $i<k$, $I\in \win(i), K\in \win(k)$, $K$ compatible with $I$. If $\ui\in |I|$ and $\uk\in |K|$ are fixed, then $\dif (M,\ui,\uk)$ is  \emph{decreasing on the subscripts}, by compatibility. Therefore
\begin{equation}
\left[\dif(M;\ui,\uk)_k,\dif(M;\ui,\uk)_i\right]\label{inter}
\end{equation} is a non--empty closed interval, denoted $\int(M;i,k;\ui,\uk)$.
{It may degenerate to   $[-\infty,a]$, $[a,+\infty]$, $[-\infty,+\infty]$ or $[a,a]$  for some real $a$.}
\end{rem}

\begin{ex}\label{ex:non_real}
Given
$$A=\left[\begin{array}{rrr}
1&-\infty&   -\infty\\
 a_{21}&a_{22}&0
\end{array}\right],\qquad B=\left[\begin{array}{rrr}
-\infty&1&   -\infty\\
 b_{21}&b_{22}&0
\end{array}\right],$$ we get
$$M=\left[\begin{array}{rrr}
1&1&   -\infty\\
 m_{21}&m_{22}&0
\end{array}\right],$$ some $a_{21}, a_{22},b_{21}, b_{22},m_{21}$ and $m_{22}\in\T$. Take $i=1$, $I=(1,2)\in\win(1)$, $k=2$ and $K=(3,3)\in\win(2)$. Then $K$ is compatible with $I$, for any $m_{21}, m_{22}$ since $m_{13}=-\infty$. Then the interval (\ref{inter}) equals $[m_{21}, +\infty]$, for $\ui=1$ and $\uk=3$.
\end{ex}

\begin{dfn} \label{dfn:halfline}

An \emph{interval relation}  is an expression  $x_i\in[a,b]+x_k$, where $a\le b$ in $\ot$ and $i,k\in [n]$. Equivalently, an interval relation  is a pair of \emph{concatenated} bivariate linear inequalities $x_k+a\le x_i\le x_k+b$. 
\end{dfn}

Expressions (\ref{eqn:rem21})--(\ref{eqn:rem24}) imply that the following, at most, \textbf{four}  \emph{interval relations} must be true:
\begin{equation}\label{eqn:inter_relation}
x_{\uk}\in \int(M;i,k;\ui,\uk)+x_\ui, \quad \ui\in |I|,\  \uk\in |K|.
\end{equation}
 Notice that
(\ref{eqn:inter_relation}) is tautological  when $\ui=\uk$.

\setcounter{thm}{2}
\begin{ex} (Continued)
In our running example,  take $i=1$, $k=2$ and $I=(1,4)$. The winning pair  $K=(1,3)$ is compatible with $I$  and this gives rise to the non--empty closed intervals
\begin{align*}
\int(M;1,2;1,1)=[0,0],&\qquad\int(M;1,2;1,3)=[1,4],\\
\int(M;1,2;4,1)=[-5,5],&\qquad \int(M;1,2;4,3)=[-4,9]
\end{align*}
and the interval relations
$$x_3\in[1,4]+x_1,\quad x_1\in[-5,5]+x_4,\quad x_3\in[-4,9]+x_4.$$
Equivalently, we can write
\begin{align*}
x_1+1\le x_3\le&\,x_1+4,\\
x_4-5\le x_1\le&\,x_4+5,\\
x_4-4\le x_3\le&\,x_4+9.
\end{align*}

Where does, say,  the interval relation $x_3\in[-4,9]+x_4$ come from? We know that $I=(1,4)\in\win(1)$ translate into (\ref{eqn:I})-(\ref{eqn:III}) and, similarly,  $K=(1,3)\in\win(2)$ translate into
\begin{align}
x_1-x_3+1&=0,\label{eqn:K}\\
-x_1+x_2+1&\le0,\label{eqn:KK}\\
-x_1+x_4-5&\le0.\label{eqn:KKK}
\end{align} These six expressions imply {the concatenated inequalities} $-4+x_4\le x_3\le 9+x_4$. This  is {possible}  by compatibility of $K$ with $I$, since the tropical minor $\left|\begin{array}{cc}
8&-1\\
1&5\\
\end{array}\right|_{trop}=13$ attains its value at the main diagonal.
\end{ex}

%

\m
 Four tropical minors of the maximum matrix $M$ must be checked out, in order to decide compatibility of $K$ with $I$. We can \emph{forget repeated minors}, keeping just one of them.  
 Any minor with repeated columns will be called  \emph{trivial}. \emph{Trivial minors will  be disregarded}; they will play no role. A minor is \emph{tropically singular} if it attains its value at both diagonals; otherwise the minor is \emph{tropically regular}. Of course, $M(i,k;\ui,\ui)$  is  tropically singular. Now, if $\ui\neq\uk$, the minor  $M(i,k;\ui,\uk)$ is tropically singular  if and only if the interval relation (\ref{eqn:inter_relation}) reduces to the non--tautological bivariate equation
 \begin{equation}\label{eqn:eqn2}
 x_{\uk}=\dif(M;\ui,\uk)_k+x_\ui\quad (=\dif(M;\ui,\uk)_i+x_\ui).
 \end{equation}

 Summing up,
if $1\le i<k\le m$, then the conditions $|I|=\{\ui_1,\ui_2\}$, $I\in\win(i)$, $|K|=\{\uk_1,\uk_2\}$, $K\in\win(k)$ and $K$ compatible with $I$  provide, at most,  two  bivariate linear equations and, at most, four interval relations,  
namely
\begin{align}
x_{\ui_2}&=\dif(M;{\ui_1},{\ui_2})_i+x_{\ui_1},\label{eqn:1}\\
x_{\uk_2}&=\dif(M;{\uk_1},{\uk_2})_k+x_{\uk_1},\label{eqn:2}\\
x_{\uk_1}&\in\int(M;i,k;{\ui_1},{\uk_1})+x_{\ui_1},\label{eqn:3}\\
x_{\uk_2}&\in\int(M;i,k;{\ui_1},{\uk_2})+x_{\ui_1},\label{eqn:4}\\
x_{\uk_1}&\in\int(M;i,k;{\ui_2},{\uk_1})+x_{\ui_2},\label{eqn:5}\\
x_{\uk_2}&\in\int(M;i,k;{\ui_2},{\uk_2})+x_{\ui_2}.\label{eqn:6}
\end{align}

The expressions (\ref{eqn:3}) and (\ref{eqn:5}) can be combined into one such, using (\ref{eqn:1}). Similarly,
(\ref{eqn:4}) and (\ref{eqn:6}) can be combined, using (\ref{eqn:2}).
In conclusion, expressions (\ref{eqn:1})--(\ref{eqn:6}) are reduced to (\ref{eqn:1}), (\ref{eqn:2})
and
\begin{align}
x_{\uk_1}&\in[a,b]+x_{\ui_1},\label{eqn:13}\\
x_{\uk_2}&\in[c,d]+x_{\ui_1},\label{eqn:14}
\end{align}
where $a,b,c,d\in\ot$, $a\le b$, $c\le d$ depend on $M$, $i,k$, $I$ and $K$.
In addition,  a solution $x$ to (\ref{eqn:igualdad}) arising from $I$ and from $K$ must satisfy the following:
\begin{equation}\label{eqn:halfline}
x_j\le (y_i-m_{ij})\oplus'(y_k-m_{kj}),
\end{equation}
for all $j\notin |I|\cup |K|\cup F(i)\cup F(k)$. This follows from (\ref{eqn:ineq}).

\m
Recall that $m$ is the number of rows of the matrix $M$.

\setcounter{thm}{10}
\begin{dfn}\label{dfn:maximal}
Let  $\Upsilon=(\suces I1m)$ be  an $m$--tuple with  $I_h\in\win(h)$, for every $h\in[m]$. We say that $I$ is a \emph{win sequence}  for (\ref{eqn:igualdad}) if  $I_h$ is compatible with $I_i$,  for all  $1\le i<h\le m$.
\end{dfn}

For a win sequence $\Upsilon=(\suces I1m)$, write $$|\Upsilon|=\bigcup_{h=1}^m\left|I_{h}\right|.$$
Given $i,j\in|\Upsilon|$, write
$i\sim j$ if there exist $k,l\in[m]$ such that $i\in|I_k|$, $j\in|I_l|$ and $|I_k|\cap |I_l|\neq\emptyset$. Closing up under transitivity,
we obtain an equivalence relation on $|\Upsilon|$.

\begin{dfn}
Let $\Upsilon=\left(I_{1},\ldots,I_{m}\right)$ be a win sequence.
\begin{enumerate}
\item An index  $i\in\left[n\right]$ is \emph{free in $\Upsilon$} if $i\notin|\Upsilon|$.
  \item An equivalence class for the relation above is called a \emph{cycle} in $\Upsilon$.
\end{enumerate}
\end{dfn}

Consider  a win sequence $\Upsilon=(\suces I1m)$.  Let  $c$ be the number of cycles in $\Upsilon$. We have $1\le c\le \card |\Upsilon|\le \min\{2m,n\}$.  After relabeling columns,
we can suppose that the cycles in $\Upsilon$ are
\begin{align*}
C_{1}&=[k_1],\\
C_{2}&=[k_2]\setminus[k_{1}],\\
&\hdots\\
C_{c}&=[k_c]\setminus[k_{c-1}],
\end{align*}
and that $[n]\setminus[k_c]$ are the free indices, for some $1\le k_1<\cdots<k_c=\card |\Upsilon|\le n$.

\begin{thm}
Each win sequence $\Upsilon=(\suces I1m)$ provides a convex set, $\sol_\Upsilon\subseteq \T^n$, of solutions to the system
(\ref{eqn:igualdad}). The set $\sol_\Upsilon$ consists of all the solutions $x\in\T^n$ arising from $I_h$, for all $h\in[m]$. Moreover,
\begin{equation}\label{eqn:dim}
\dim\left(\sol_\Upsilon\right)\le n-\card|\Upsilon|+c.
\end{equation}
All solutions  to
(\ref{eqn:igualdad}) are obtained this way.
\end{thm}


\begin{proof}
The last statement follows from remark \ref{rem:compatible}.  Convexity of $\sol_\Upsilon$ is trivial, because $\sol_\Upsilon$ is the set of $x\in\T^n$  which satisfy a certain system of classical linear equations and of inequalities: these come up from (\ref{eqn:ineq}), for $I=I_h$, $h\in[m]$. Of course, the dimension of $\sol_\Upsilon$ is the dimension of its linear hull.

\m
For the bound on the dimension, with the notations above, notice that
\begin{itemize}
\item Each cycle $C$ provides a system $T_C$ of
bivariate linear inequalities (these are related to certain interval relations) only involving $x_j$, with $j\in C$ (see example 3 in p. \pageref{ex:current_fin}). This can make the dimension  decrease (when an interval collapses to a point) or stay the same.
\item Each cycle $C$ provides a system $S_C$ of
bivariate linear equations only involving $x_j$, with $j\in C$ (see example 3 in p. \pageref{ex:current_fin}). There are two possible cases: either  the only solution to $S_C$ is trivial (i.e., $S_C$ is incompatible over $\R$; see example \ref{ex:empty}) or  the solution set to $S_C$ is one--dimensional. Notice that, in the former case, the condition $x_j=-\infty$,  for some $j\in C$ (together with certain additional inequalities) might imply $x_l=-\infty$ for some free index $l$ (this explains the expression \lq\lq at most", in the next item).
\item Each free index  $l$ makes the dimension increase by one unit, \emph{at most}. Notice that  
relations such as (\ref{eqn:halfline}) can make the dimension decrease (when $y_i=-\infty$ or $y_k=-\infty$) or stay the same.
\end{itemize}
There are $c$ cycles and  $n-\card |\Upsilon|$ free indices, so that formula (\ref{eqn:dim}) follows.
\end{proof}

\begin{ex}\label{ex:empty}
The set $\sol_\Upsilon$ can be trivial. Indeed, for 
 $$A=\left[\begin{array}{rrrr}
 3 &    7&    -1&  -\infty \\
     6&     7&  -\infty &  -\infty \\
     -9&     0&     0&  -\infty \\
\end{array}\right], \quad
B=\left[\begin{array}{rrrr}
-\infty &  -\infty &  -\infty &     8\\
  -\infty &  -\infty &     5&     1\\
     -9&     0&     -\infty&    -4\\
\end{array}\right],$$ we get
$$M=\left[\begin{array}{rrrr}
3&     7&    -1&     8\\
     6&     7&     5&     1\\
     -9&     0&     0&     -4\\
\end{array}\right].$$
Then $\Upsilon=((1,4), (1,3), (3,4))$ is a win sequence which implies
\begin{align*}
x_1-x_4-5&=0,\\
x_1-x_3+1&=0,\\
x_3-x_4+4&=0,
\end{align*}
but this linear system is incompatible over $\R$. Over $\T$, the only solution is $x_1=x_3=x_4=-\infty$. In addition, $x_2=-\infty$ because, from the first row of $A$ we know that $x_2+7\le x_1+3$, (i.e., $x_2-x_1+4\le0$, in normal form); {see remove and enlarge in p. \pageref{erasing}. Here, $\Omega=[4]$.}
\end{ex}

If no win sequences exist, then the only solution to the system (\ref{eqn:igualdad}) is trivial.
The number $p$ of win sequences is no bigger that  $r^m$, where $r=\max\{\lceil\frac{n}{2}\rceil\lfloor\frac{n}{2}\rfloor,n\}$.
Even if some win sequence does exist, it may happen that the only solution to the system (\ref{eqn:igualdad}) is trivial.

\setcounter{thm}{2}
\begin{ex} (Continued)
Let us finish our current example.
The win sequences are $\Upsilon_1=((1,4),(1,3),(3,3))$ and $\Upsilon_2=((2,4),(1,3),(3,3))$.

For $\Upsilon_1$ we must solve the systems
$$S=S_{\Upsilon_1}:\left\{\begin{aligned} x_1+3&=x_4+8,\\ x_1+6&=x_3+5,\\ x_3+1&=x_3+1,\end{aligned}\right.$$
$$T=T_{\Upsilon_1}:\left\{\begin{aligned}
x_2+7&\le x_1+3,\\
x_3-1&\le x_1+3 ,\\
x_2+7&\le x_1+6,\\
x_4+1&\le x_1+6,\\
x_1+1&\le x_3+1,\\
x_2+0&\le x_3+1,\\
x_4+2&\le x_3+1.\end{aligned}\right.$$
Writing in normal form, the system $S$ leads to the coefficients matrix \label{ex:current_fin}
\begin{equation}\label{mat:C1}
C=\left[
\begin{array}{rrrrr}
1&0&0&-1&-5\\
1&0&-1&0&1
\end{array}
\right]
\end{equation}
By Gaussian elimination (which dates back, at least,  to the mathematicians of ancient China, in the 2nd or 3rd century B.C.,   and should better  be called the \emph{Chinese elimination}; see \cite{Boyer} p.~219), we transform the former matrix into the following upper triangular matrix (also denoted by $C$)
\begin{equation}\label{mat:C2}
C=\left[
\begin{array}{rrrrr}
1&0&-1&0&1\\
0&0&1&-1&-6
\end{array}
\right].
\end{equation}
This gives the following partial result: $x_3=x_4+6$ and $x_1=x_3-1$. Substituting $x_3$ by its value (this is usually called \emph{backward substitution}), we get
\begin{equation}\label{eqn:sol}
x_2,x_4\in\T,\quad x_3=x_4+6,\quad x_1=x_4+5.
\end{equation}
In particular,  the indices $1,3,4$ belong to the same cycle.

Writing in normal form and removing tautologies, the system $T$ yields the following coefficient matrix
\begin{equation}\label{mat:D1}
D=\left[
\begin{array}{rrrrr}
-1&1&0&0&4\\
-1&0&1&0&-2\\
-1&1&0&0&1\\
-1&0&0&1&-5\\
0&1&-1&0&-1\\
0&0&-1&1&1
\end{array}
\right].
\end{equation}

We want to simplify the information contained in the matrix $D$ as much as possible. In order to do so we will produce two matrices $E$, $N$ such that
$$D[x,1]^T\le0\ \Leftrightarrow\ E[x,1]^T=0\ \text{and}\ N[x,1]^T\le0,$$ $$2\card \rows (E)+\card \rows (N)\le \card \rows (D).$$
Either matrix  $E$ or $N$ could be empty. We will denote $N$  by $D$. {The desired simplification occurs whenever  $\card \rows (N)< \card \rows (D).$}
\m

In order to  simplify, we allow the following \textbf{row operations} \label{dfn:row_oper} on {the real matrix} $D$:
\begin{enumerate}
\item permute two rows,
\item replace a row by the sum of that row plus another row, so that the resulting row has, at most, three non--zero entries,
\item remove a row $r$, when it is \emph{superfluous},  in the following sense: if $r=(\suces r1n,a)$ and $s=(\suces r1n,b)$ are  rows in $D$,  with $r_j\in\R$ and $a\le b\in\R$, then $r$ is superfluous,
\item remove two rows in $D$, when they are opposite, and write any one of them into the matrix $E$.
\end{enumerate}

Warning:  multiplication of a  row by a negative number is not allowed here.

By means of such row operations, we can achieve  an \lq\lq almost" upper triangular matrix $D$ (the precise term is \emph{sub--special}; see definition \ref{dfn:special} below)  and a matrix $E$ (but $E$ turns out to be empty, in this particular case):
\begin{equation}\label{mat:D2}
D=\left[\begin{array}{rrrrr}
-1&1&0&0&4\\
-1&0&1&0&-2\\
-1&0&0&1&-5\\
0&1&-1&0&-1\\
0&0&-1&1&1
\end{array}\right]
\end{equation}
Notice that row $(-1,1,0,0,1)$ in (\ref{mat:D1}) has been eliminated because row  $(-1,1,0,0,4)$ makes it superfluous (the meaning is that $-x_1+x_2+4\le0$ implies $-x_1+x_2+1\le0$).
Now, the matrix (\ref{mat:D2}) stands for the system
\begin{align*}
-x_1+x_2+4\le&0,\\
-x_1+x_3-2\le&0,\\
-x_1+x_4-5\le&0,\\
x_2-x_3-1\le&0,\\
-x_3+x_4+1\le&0.
\end{align*}
We substitute $x_3$ and $x_1$ by their values, as shown in (\ref{eqn:sol}), obtaining a new system of inequalities, which we must write in normal form and apply row operations to it again. Finally, we obtain:
\begin{equation}\label{eqn:sol2}
x_2-x_4-1\le0,
\end{equation} corresponding to the matrix
\begin{equation}
D=\left[\begin{array}{rrrrr}
0&1&0&-1&-1
\end{array}\right].
\end{equation}
Summing up, the solutions arising from $\Upsilon_1$ are given by (\ref{eqn:sol}) and (\ref{eqn:sol2}):
$$x=\left[\begin{array}{c}
x_4+5\\
   x_2\\
 x_4+6\\
   x_4\\
\end{array}\right],\quad\text{s.t.\ } x_2-x_4-1\le0.$$

Similarly, we get that the solutions arising from $\Upsilon_2$ are
$$x=\left[\begin{array}{c}
x_3-1\\
   x_4+1\\
 x_3\\
   x_4\\
\end{array}\right],\quad\text{s.t.\ } x_3-6\le x_4\le x_3-3.$$
\end{ex}

The former example is rather general: we can apply the same procedure to find the solutions $\sol_\Upsilon$, for any win sequence $\Upsilon$.

\m

It is now convenient to give names to certain types of real matrices that we will deal with.

\setcounter{thm}{14}
\begin{dfn}\label{dfn:special} Let  $G=(g_{ij})\in\MM_{m\times (n+1)}(\R)$.
\begin{enumerate}
\item Let $G'\in\MM_{m\times n}(\R)$ be obtained from $G$ by deleting the last column.
\item The  matrix  $G$ is  \emph{special} if   each row of $G'$ is  a permutation of the $n$--vector $(1,-1,0,\ldots,0)$.

\item The special matrix $G$ is  \emph{super--special} if
the first non--zero entry of each row is 1.

\item The special matrix $G$ is \emph{sub--special} if
\begin{enumerate}
\item all rows in $G$ are different and different from rows in $-G$,
\item all rows in $G'$ are different,
\item if $\row(G',i)=-\row(G',k)$, for some $i<k$, then  $k=i+1$, $$g_{i,n+1}<-g_{i+1,n+1}$$  and $\row(G',i)=(\overbrace{0,\ldots,0}^{j-1},1,\overbrace{0,\ldots,0}^{l-1},-1,0,\ldots,0)$, for some $j,l\in [n]$, \label{con:inter}
\item if $\row(G',i)\neq -\row(G',i+1)$, for some $i$, then   $\min\{j: g_{ij}\neq0\}\le\min\{j: g_{i+1, j}\neq0\}$.
\end{enumerate}
\end{enumerate}
\end{dfn}

\begin{ex} Matrices (\ref{mat:C2})--(\ref{mat:D2})  are special.
The following matrix  $G$ is sub--special but not super--special
$$\left[ \begin{array}{rrrrr}1&0&-1&0&3\\ -1&0&1&0&-8\\ 0&-1&1&0&-4\\ 0&0&1&-1&0
\end{array}\right].$$
\end{ex}

%

\begin{rem}
Notice that, if $G$ is a non--empty sub--special matrix, then the set of $x\in\T^n$ such that $G[x,1]^T\le0$ is neither trivial nor $\T^n$. The set of all such $x$ will be denoted $\sol_G$. If $G$ is empty, we define $\sol_G$ as $\T^n$.
\end{rem}

\begin{rem}\label{rem:sub_sp}
\begin{itemize}
\item By  Gaussian elimination, any special matrix $C$ \emph{corresponding to the system $S$ of equalities $C[x,1]^T=0$}, can  be transformed into a super--special upper triangular matrix {$E$,  such that $$C[x,1]^T=0\ \Leftrightarrow\ E[x,1]^T=0,$$
    $$\card \rows (E)\le \card \rows (C).$$} $E$ might be empty. The matrix $E$ is relabeled as $C$.
\item By  the row operations on p.~\pageref{ex:current_fin}, any special matrix $D$ \emph{corresponding to a system  $T$ of inequalities $D[x,1]^T\le0$}, can  be transformed into a sub--special upper triangular one  $N$ and an additional special matrix $E$, with  $$2\card \rows (E)+\card \rows (N)\le \card \rows (D).$$ Either matrix $N$ or $E$ might be empty. The matrix $E$ corresponds to a system of equalities derived from $T$. The matrix $E$ can be transformed into a super--special one. The matrix $N$ is relabeled as $D$.
    This procedure is called \textbf{sub--specialization of $D$}.
\end{itemize}
\end{rem}

\m
%
%


\section{The algorithm}

An algorithm to solve (\ref{eqn:igualdad}) must find first all win sequences. Then,  for each win sequence $\Upsilon$, the algorithm must write two bivariate systems: one system $S_\Upsilon$ of equations and one system $T_\Upsilon$ of inequalities. Writing these systems into normal form and performing \textbf{remove and enlarge} as in p. \pageref{erasing}, the algorithm must  compute a subset {$\Omega_\Upsilon\subseteq[n]$} and special matrices $C_\Upsilon, D_\Upsilon\in\MM_{m\times (n+1)}(\R)$ such that the former systems are equivalent to
\begin{equation}
C_\Upsilon [x,1]^T=0, \qquad D_\Upsilon [x,1]^T\le 0, {\qquad x_j=-\infty,\ j\in\Omega_\Upsilon.}\label{eqn:igu_des}
\end{equation}
The set of solutions to (\ref{eqn:igu_des}), denoted $\sol_\Upsilon=\sol_{C_\Upsilon}\cap\sol_{D_\Upsilon}$, must now be computed. It might be trivial.
 The non--trivial solutions to (\ref{eqn:igualdad}) is the union of $\sol_\Upsilon$, as $\Upsilon$ runs over all win sequences.

\centerline{\textbf{ALGORITHM}}
\begin{itemize}
\item STEP 1:  compute  the matrices $\textbf{A},\textbf{B}$ and $M$. Replace $A$ and $B$ by $\textbf{A}$ and $\textbf{B}$.
\item STEP 2:  Compute  all winning pairs, for all $i\in [m]$. Store them in a tridimensional array  $W$ ($r$ rows, 2 columns, $m$ pages).  In page $i$ we store all members of $\win(i)$. Blanks are 
    padded with zeros.
\item STEP 3: Compute all win sequences. Store them in a tridimensional array $WS$ ($m$ rows, 2 columns, $p$ pages), with $0\le p\le r^m$. No entry of $WS$ is zero. If $WS$ is empty, then the only solution to  system (\ref{eqn:igualdad}) is trivial, RETURN.

\item  FOR each win sequence $\Upsilon$
\begin{itemize}

\item STEP 4:  Compute the set {$\Omega_\Upsilon\subseteq [n]$ and} the special matrices $C_\Upsilon$ and  $D_\Upsilon$.

\item STEP 5: Sub--specialize matrix  $D_\Upsilon$. In order to do so, first, obtain special matrices $E_\Upsilon$, $N_\Upsilon$ such that
$$D_\Upsilon[x,1]^T\le0\ \Leftrightarrow\ E_\Upsilon[x,1]^T=0\ \text{and}\ N_\Upsilon[x,1]^T\le0,$$ $$2\card \rows (E_\Upsilon)+\card \rows (N_\Upsilon)\le \card \rows (D_\Upsilon).$$
Either matrix  $E_\Upsilon$ or $N_\Upsilon$ could be empty. Denote $N_\Upsilon$  by $D_\Upsilon$. By row operations,  work on  $D_\Upsilon$ to make it  sub--special.


\item STEP 6: Concatenate the matrices $C_\Upsilon$ and $E_\Upsilon$ into a matrix, which we can denote again by $C_\Upsilon$. By Gaussian elimination,  work on  $C_\Upsilon$ to make it  super--special and upper triangular. {The number of rows of $C_\Upsilon$ may decrease.}
Solve the  classical linear system $C_\Upsilon[x,1]^T=0$. {The set $\Omega_\Upsilon$ may enlarge.} The solution set in $\T^n$ is denoted $\sol_{C_\Upsilon}$; it may depend on a number of parameters. If $\sol_{C_\Upsilon}$ is trivial, GO TO WORK WITH THE NEXT WIN SEQUENCE.

\item STEP 7:  If $D_\Upsilon$ is non--empty, substitute $x\in\sol_{C_\Upsilon}$ into $D_\Upsilon[x,1]^T\le0$ to obtain a new system of linear inequalities. Write this system in normal form and  denote it $D_\Upsilon[x,1]^T\le0$ again. {The set $\Omega_\Upsilon$ may enlarge.}

\item STEP 8: If $D_\Upsilon$ is non--empty, sub--specialize matrix  $D_\Upsilon$. Obtain special matrices $E_\Upsilon$, $N_\Upsilon$ such that
$$D_\Upsilon[x,1]^T\le0\ \Leftrightarrow\ E_\Upsilon[x,1]^T=0\ \text{and}\ N_\Upsilon[x,1]^T\le0,$$ $$2\card \rows (E_\Upsilon)+\card \rows (N_\Upsilon)\le \card \rows (D_\Upsilon).$$
Either matrix  $E_\Upsilon$ or $N_\Upsilon$ could be empty. Denote $N_\Upsilon$  by $D_\Upsilon$. By row operations,  work on  $D_\Upsilon$ to make it  sub--special.
  If $E_\Upsilon$ is empty, then
$$\sol_\Upsilon=\sol_{C_\Upsilon}\cap\sol_{D_\Upsilon}.$$ Otherwise, GO TO STEP 6.
\end{itemize}

\item END FOR

All the solutions to (\ref{eqn:igualdad}) are $\bigcup_{\Upsilon\in WS}\sol_\Upsilon$.
\end{itemize}

We have programmed the former algorithm to solve  system (\ref{eqn:igualdad}).
 Working over $\Q$, let us  compute the complexity of it. The arithmetic complexity counts the number of arithmetic operations ($+,-,\max,\min, <,=$ and $>$, in our situation) in the worst possible case.

 Our programme is divided into two parts. In the first part, we determine all the win sequences.   Say we get $p$  win sequences. The arithmetical complexity of this part is  $$O(m^2n^3p).$$  In the second part, we compute the matrices $C_\Upsilon, D_\Upsilon$ and all the solutions (if any),  for each win sequence $\Upsilon$.  The arithmetic complexity  of the second part is  $$O(m(m^2+n)p).$$ Since the maximum number of winning pairs is $r=\max\{\lceil\frac{n}{2}\rceil\lfloor\frac{n}{2}\rfloor,n\} $, then $p\le r^m$,  where $r$ is $O(n^2)$. This  gives an exponential arithmetical complexity! But, let us take a closer look.
 Clearly, the bigger $n$, the \textbf{more} winning pairs we may have, for each $i\in[m]$. On the other hand, the bigger $m$, the \textbf{fewer} win sequences we have, in  probability, due to the compatibility requirement. Indeed, given winning pairs $I\in\win(i)$, $K\in\win(k)$ with $1\le i<k\le m$, let us assume that the probability of $K$ being compatible with $I$ as $1/2$ (this assumption  is rather reasonable, since \lq\lq $K$ is compatible with $I$" is a yes/no event). Thus, given any sequence of pairs $\Upsilon=(\suces I1m)$,   the probability of $\Upsilon$ being a win sequence is, roughly, $$\frac{1}{2^{m\choose2}}\sim \frac{1}{2^{m^2}}.$$ This proves that if $m$ is big, then we expect $p$ rather small.
In particular, the worst case, $p=r^m$, is  unlikely to happen.
With this in mind, an average complexity  (see \cite{Bogdanov}) for the first part is $$O(m^2n^{3+2m}/2^{m^2})=O(m^22^{(3+2m) \log_2 n-m^2})$$ and it will be, at most polynomial $O(m^2)$, if $\log_2 n\le\frac{m^2}{3+2m}\sim\frac{m}{2}$.
For  the second part we get two terms:  $$O(m^3n^{2m}/2^{m^2})=O(m^32^{2m\log_2n-m^2})$$ and it will be, at most polynomial $O(m^3)$, if $\log_2 n\le\frac{m}{2}$, and
$$O(mn^{1+2m}/2^{m^2})=O(m2^{(1+2m)\log_2 n-m^2})$$ and it will be, at most polynomial $O(m)$, if $\log_2 n\le\frac{m^2}{1+2m}\sim\frac{m}{2}$.

\section{Some examples}

\begin{ex} Given
$$A=\left[\begin{array}{rrr}
  1&     3&  -\infty \\
     5&     0&  -\infty \\
  -\infty&      3&  -\infty
\end{array}\right], \quad
B=\left[\begin{array}{rrr}
-\infty&   -\infty&      3\\
     5&     0&     2\\
     3&  -\infty&      2
\end{array}\right],$$ we get
$$M=\left[\begin{array}{rrr}
1&     3&     3\\
     5&     0&     2\\
     3&     3&     2
\end{array}\right].$$
The only win sequence is  $\Upsilon=((2,3),(1,1),(2,1))$.
The solutions arising from $\Upsilon$ are
$$x=\left[\begin{array}{c}
x_3\\
x_3\\
x_3\\
\end{array}\right].$$
Here we have  $m=n=r=3$, $r^m=27$, $r^m/2^{m^2}\simeq 0.0527$ and $p=1$. The programme execution lasted $0.038883$ seconds with Matlab R2007b in a portable PC working with Windows Vista.
\end{ex}

\begin{ex} (From \cite{Sergeev_Wagneur}) Given
$$A=\left[\begin{array}{rrrrrrr}
-\infty&   -\infty&   -\infty&      0&     4&     2&     6\\
  -\infty&      5&     6&  -\infty&   -\infty&   -\infty&      2
\end{array}\right],$$
$$B=\left[\begin{array}{rrrrrrr}
 0&     1&     5&  -\infty&   -\infty&   -\infty&   -\infty \\
     3&  -\infty&   -\infty&      0&     2&     4&  -\infty
\end{array}\right],$$ we get
$$M=\left[\begin{array}{rrrrrrr}
0&     1&     5&     0&     4&     2&     6\\
     3&     5&     6&     0&     2&     4&     2
\end{array}\right].$$
The win sequences are $\Upsilon_1=((4,1),(2,1))$, $\Upsilon_2=((4,3),(2,1))$, $\Upsilon_3=((5,1),(2,1))$, $\Upsilon_4=((5,3),(2,1))$, $\Upsilon_5=((6,1),(2,1))$, $\Upsilon_6=((6,3),(2,1))$, $\Upsilon_7=((7,1),(2,1))$ and $\Upsilon_8=((7,3),(2,1))$.

Here we have  $m=2$, $n=7$, $r=12$, $r^m=144$, $r^m/2^{m^2}=9$ and $p=8$. The programme execution lasted $1.210653$ seconds with Matlab R2007b in a portable PC working with Windows Vista.

The solutions arising from $\Upsilon_1$ are
$$x=\left[\begin{array}{c}
x_4\\
x_4-2\\
x_3\\
x_4\\
x_5\\
x_6\\
x_7
\end{array}\right],\quad\text{s.t.\ }
\begin{aligned}
x_3-x_4+5&\le0,\\
-x_4+x_5+4&\le0,\\
-x_4+x_6+2&\le0,\\
-x_4+x_7+6&\le0.
\end{aligned}$$

The solutions arising from $\Upsilon_2$ are
$$x=\left[\begin{array}{c}
x_2+2\\
x_2\\
x_4-5\\
x_4\\
x_5\\
x_6\\
x_7
\end{array}\right],\quad\text{s.t.\ }
\begin{aligned}
x_2+2&\le x_4\le x_2+4,\\
-x_2+x_5-3&\le0,\\
-x_2+x_6-1&\le0,\\
-x_2+x_7-3&\le0,\\
-x_4+x_5+4&\le0,\\
-x_4+x_6+2&\le0,\\
-x_4+x_7+6&\le0.
\end{aligned}$$

The solutions arising from $\Upsilon_3$ are
$$x=\left[\begin{array}{c}
x_5+4\\
x_5+2\\
x_3\\
x_4\\
x_5\\
x_6\\
x_7
\end{array}\right],\quad\text{s.t.\ }
\begin{aligned}
x_3-x_5+1&\le0,\\
x_4-x_5-4&\le0,\\
-x_5+x_6-2&\le0,\\
-x_5+x_7+2&\le0.
\end{aligned}$$

The solutions arising from $\Upsilon_4$ are
$$x=\left[\begin{array}{c}
x_2+2\\
x_2\\
x_5-1\\
x_4\\
x_5\\
x_6\\
x_7
\end{array}\right],\quad\text{s.t.\ }
\begin{aligned}
x_2-2&\le x_5\le x_2,\\
x_4-x_5-4&\le0,\\
-x_2+x_4-5&\le0,\\
-x_2+x_6-1&\le0,\\
-x_2+x_7-3&\le0,\\
-x_5+x_6-2&\le0,\\
-x_5+x_7+2&\le0.
\end{aligned}$$

The solutions arising from $\Upsilon_5$ are
$$x=\left[\begin{array}{c}
x_6+2\\
x_6\\
x_3\\
x_4\\
x_5\\
x_6\\
x_7
\end{array}\right],\quad\text{s.t.\ }
\begin{aligned}
x_3-x_6+3&\le0,\\
x_4-x_6-2&\le0,\\
x_5-x_6+2&\le0,\\
-x_6+x_7+4&\le0.
\end{aligned}$$

The solutions arising from $\Upsilon_6$ are
$$x=\left[\begin{array}{c}
x_2+2\\
x_2\\
x_6-3\\
x_4\\
x_5\\
x_6\\
x_7
\end{array}\right],\quad\text{s.t.\ }
\begin{aligned}
x_2&\le x_6\le x_2+1,\\
x_4-x_6-2&\le0,\\
x_5-x_6+2&\le0,\\
-x_2+x_4-5&\le0,\\
-x_2+x_5-3&\le0,\\
-x_2+x_7-3&\le0,\\
-x_6+x_7+4&\le0.
\end{aligned}$$

The solutions arising from $\Upsilon_7$ are
$$x=\left[\begin{array}{c}
x_7+6\\
x_7+4\\
x_3\\
x_4\\
x_5\\
x_6\\
x_7
\end{array}\right],\quad\text{s.t.\ }
\begin{aligned}
x_3-x_7-1&\le0,\\
x_4-x_7-6&\le0,\\
x_5-x_7-2&\le0,\\
x_6-x_7-4&\le0.
\end{aligned}$$

The solutions arising from $\Upsilon_8$ are
$$x=\left[\begin{array}{c}
x_2+2\\
x_2\\
x_7+1\\
x_4\\
x_5\\
x_6\\
x_7
\end{array}\right],\quad\text{s.t.\ }
\begin{aligned}
x_2-4&\le x_7\le x_2-2,\\
x_4-x_7-6&\le0,\\
x_5-x_7-2&\le0,\\
x_6-x_7-4&\le0,\\
-x_2+x_4-5&\le0,\\
-x_2+x_5-3&\le0,\\
-x_2+x_6-1&\le0.
\end{aligned}$$
The set of solutions has pure dimension 5.
\end{ex}

\section*{Acknowledgements} Our programme is written in MATLAB. Our gratitude to R. Gonz\'{a}lez Alberquilla, for answering a few questions about MATLAB,  and to J.P. Quadrat for having drawn our attention to this problem for which, we believe, new algorithms are still of interest.


\end{document}